\documentclass[12pt]{article}
\usepackage{amsmath}
\usepackage[usenames]{color}
\usepackage{mathrsfs}
\usepackage{amsfonts}
\usepackage{amssymb,amsmath}
\usepackage{CJK}
\usepackage{cite}
\usepackage{cases}
\usepackage{amsthm}

\baselineskip 5.5pt \lineskip 5.5pt \numberwithin{equation}{section}

\def\Der{{\rm Der}}
\def\Inn{{\rm Inn}}
\def\Aut{{\rm Aut}}

\def\ad{\mbox{ad}}

\def\vs{\vspace*}
\def\cl{\centerline}

\def\cl{\centerline}

\def\C{\mathbb{C}}
\def\Z{\mathbb{Z}}

\def\0{\overline{0}}
\def\1{\overline{1}}

\def\L{{\mathfrak{L}^s_\lambda}}

\newtheorem{theo}{Theorem}[section]

\newtheorem{lemm}[theo]{Lemma}

\begin{document}

\cl{\bf{Derivation algebras and automorphism groups}}
\cl{{\bf{of a class of deformative super $W$-algebras $W^s_\lambda(2,2)$}}
\footnote{Supported by NSF grants 11101056, 11271056, 11431010, 11371278 of China.}}

\cl{Hao Wang$^{1)}$, Huanxia Fa$^{2)}$, Junbo Li$^{2)}$}

\cl{\small $^{1)}$Wu Wen-Tsun Key Laboratory of Mathematics and School of Mathematical Sciences,}
\cl{\small University of Science and Technology of China, Hefei 230026, China}

\cl{\small $^{2)}$School of Mathematics and Statistics, Changshu Institute of Technology, Changshu 215500, China}

\vs{6pt}

\noindent{\bf{Abstract.}}
In this paper, the derivation algebras and automorphism groups of a class of deformative super $W$-algebras $W_{\lambda}(2,2)$ are determined.


\noindent{\it{MR(2000) Subject Classification}: }\vs{18pt} 17B05, 17B40, 17B65, 17B68.

\section{Introduction}

\indent\ \ \ \ \,\,It is well known that the $W$-algebra $W(2,2)$ plays important rolls in many areas of theoretical physics and mathematics which was introduced in \cite{ZD} for the study of vertex operator algebras generated by weight $2$ vectors. It has a basis $\{L_{m}, I_{m}\,|\,m\in\Z\}$ as a vector space over the complex field  $\C$, with the non-trivial Lie brackets $[L_{m},L_{n}]=(m-n)L_{m+n}$ and $[L_{m},I_{n}]=(m-n)I_{m+n}$. Structures and representations of $W(2,2)$ are extensively investigated in the known references, such as  \cite{CL}, \cite{GJP},  \cite{JP}, \cite{LGZ}, \cite{LSX}, \cite{ZD}, \cite{ZT} and the corresponding references.

Some Lie superalgebras whose even parts are the $W(2,2)$ Lie algebras are constructed in \cite{WCB}. The {\it deformative super $W$-algebras $W^s_\lambda(2,2)$} (denoted by $\L$ just for convenience in the following) investigated in this paper own the same even parts: the $W$-algebra $W(2,2)$ and are infinite-dimensional Lie superalgebras over the complex field $\C$ with the basis $\{L_m,\,I_m,\,G_p,\,H_p\,|\,m\in\Z,\,p\in s\!+\!\Z\}$, admitting the following non-vanishing super brackets:
\begin{eqnarray*}
\begin{array}{lll}
&[L_m,L_n]=(m-n)L_{m+n},&[L_m,I_n]=(m-n)I_{m+n},\vs{6pt}\\
&[L_m,H_p]=(\frac{m}{2}-p)H_{p+m},&[G_p,G_q]=I_{p+q},\vs{6pt}\\
&[L_m,G_p]=(\frac{m}{2}-p)G_{p+m}+\lambda(m+1)H_{m+p},&[I_m,G_p]=(m-2p)H_{p+m},
\end{array}
\end{eqnarray*}
where $s=0$ or $\frac{1}{2}$ and $\lambda\in\C$.

It is known that derivations and automorphisms play important roles in the investigation of structures and representations of the relevant Lie algebras or superalgebras. Many references (i.e., \cite{BM},\cite{DZ},\cite{F01},\cite{F02},\cite{FG},\cite{GJP},\cite{SJ},\cite{SZ},\cite{WLX}) have focused on derivations and automorphisms of different Lie algebras or superalgebras backgrounds. In this article, the corresponding derivation algebras and automorphism groups are respectively determined in Theorems \ref{thm-derivation algebra} and \ref{theorem-automorphism}.

We briefly recall some notations on Lie superalgebras. Let $L=L_{\0}\oplus L_{\1}$ be a vector space over $\C$, and all elements below are assumed to be $\Z_2$-homogeneous in this subsection, where $\Z_2=\{\overline{0},\overline{1}\}$.  For $x\in L$, we always denote $[x]\in\Z_2$ to be its \textit{parity}, i.e., $x\in L_{[x]}$. An $\Z_2$-homogeneous linear map $d: L\longrightarrow L$ satisfying $d(L_{i})\in L_{i+[d]}$ for $i\in\Z_2$ and
\begin{eqnarray}\label{eq-derivation}
d([x,y]) = [d(x),y]+(-1)^{([d][x])}[y,d(x)],\ \ x,\ y\in L,
\end{eqnarray}
is called a homogeneous derivation of parity $[d]\!\in\!\Z_2$. The derivation $d$ is called even if $[d]=\0$ and odd if $[d]=\1$. Denote by $\Der_{p}(L)$ the set of homogeneous derivations of parity $p$ ($p\in\{\0,\1\}$). Denote $\Der(L) =\Der_{\0}(L)\oplus\Der_{\1}(L)$ the set of derivations from $L$ to $L$.

Denote $\Inn(L) =\Inn_{\0}(L)\oplus\Inn_{\1}(L)$ the set of all inner derivations from $L$ to $L$, where $\Inn_{p}(L)$ is the set of homogeneous inner derivations of parity $p$ consisting of $\ad x$ $(x\in L_{p})$ defined by:
\begin{eqnarray}\label{eq-innerderivation}
 \ad x(y) = [x,y],\ \ \ y\in L,\ [x] = p.
\end{eqnarray}

Denote ${\Z_s}=\Z\cup(s+\Z)$, i.e., ${\Z_s}=\Z$ for $s=0$ and $\frac{1}{2}\Z$ for $s=\frac{1}{2}$. A Lie superalgebra $L$ is called \textit{${\Z_s}$-graded} if $L=\oplus_{r\in{\Z_s}}L_{r}$ and $[L_{p},L_{q}]\subset L_{p+q}$. It is easy to see that the algebras $\L$ considered here are ${\Z_s}$-graded with ${(\L)}_{k}=\C L_{k}\oplus\C I_{k}\oplus\C G_{k}\oplus\C H_{k}$ $(k\in\Z)$ for the case $s=0$ and ${(\L)}_{k} =\C L_{k}\oplus\C I_{k}$, ${(\L)}_{\frac{1}{2}+k} =\C G_{\frac{1}{2}+k}\oplus\C H_{\frac{1}{2}+k}$ $(k\in\Z)$ for the case $s=\frac{1}{2}$.

Denote
\begin{eqnarray}\label{def-graded-derivation}
\Der_{r}({\L})=\{d\in\Der({\L})|d(({\L})_{p})
\subset{(\L)}_{p+r},\ \forall\,\,p\in{\Z_s}\},
\end{eqnarray}
in which $r\in{\Z_s}$. An element $d\in\Der_{r}({\L})$ is called \textit{a homogeneous derivation of degree $r$}, usually denoted as $d_{r}$. Similarly, one can give the definition of $\Inn_{r}({\L})$ and \textit{a homogeneous inner derivation of degree $r$}. Then $\Der({\L})\!=\!\mbox{$\prod\limits_{r\in {\Z_s}}$}\Der_{r}({\L})$. It should be noticed that for any $d=\sum_{r\in {\Z_s}}d_{r}\in\Der({\L})$, the formal sum on the right hand side may be infinite while for any $x\in {\L}$, $d(x)=\sum_{r\in {\Z_s}}d_{r}(x)$, the nonzero summands on the right hand side must be finite since ${\L}$ are vector spaces.

A bijective linear map $\varphi$: ${\L}\rightarrow{\L}$ is called an automorphism of ${\L}$ if it satisfies:
\begin{eqnarray}\label{eq-automorphism}
\varphi([x,y])=[\varphi(x),\varphi(y)],
\ \varphi({(\L)}_{\overline{i}})\subset{(\L)}_{\overline{i}}.
\end{eqnarray}
for any $x, y \in \L, i\in\{0, 1\}$.
Denote by $\Aut({\L})$ the set of all automorphisms of ${\L}$. It is easy to verify $\Aut({\L})$ is a group with the multiplication defined by the composition of linear maps.

\section{Derivation algebra}

\indent\ \ \ \ \ \
The derivation algebra of ${\L}$ can be formulated as the following theorem.
\begin{theo}\label{thm-derivation algebra}
$\Der({\L})=\Inn({\L})\oplus\C\partial_{2}\oplus\C\partial_{4}\oplus\delta_{\lambda,0}\C\partial_{1}\oplus\delta_{\lambda,0}\C\partial_{3}$,
where $\partial_{1}$, $\partial_{2}$ are defined by \eqref{partial{1},{2}} and $\partial_{3}$, $\partial_{4}$ are defined by \eqref{partial34}.
\end{theo}

\noindent{\it Proof of Theorem \ref{thm-derivation algebra}}\ \ \,This theorem follows from Lemmas \ref{lemma(0,0,k)} to \ref{lemma(0.5,1,0)}.\hfill$\Box$

\begin{lemm}\label{lemma(0,0,k)}
$\Der_{\0}({\mathfrak{L}^0_\lambda})\cap\Der_{k}({\mathfrak{L}^0_\lambda})
=\Inn_{\0}({\mathfrak{L}^0_\lambda})\cap\Inn_{k}({\mathfrak{L}^0_\lambda})$, $\forall\,\,k\in\Z^*$.
\end{lemm}
\begin{proof}

For any $d_{\0,k}\in\Der_{\0}({\mathfrak{L}^0_\lambda})\cap\Der_{k}({\mathfrak{L}^0_\lambda})$, we can suppose
\begin{eqnarray}\label{eq(0,0,k)}
\begin{array}{ccccc}
d_{\0,k}(L_{n})\!\!\!&=&\!\!\!a_{L_{n},k}L_{n+k}\!\!\!&+&\!\!\!b_{L_{n},k}I_{n+k},\vs{6pt}\\
d_{\0,k}(I_{n})\!\!\!&=&\!\!\!a_{I_{n},k}L_{n+k}\!\!\!&+&\!\!\!b_{I_{n},k}I_{n+k},\vs{6pt}\\
d_{\0,k}(G_{n})\!\!\!&=&\!\!\!c_{G_{n},k}G_{n+k}\!\!\!&+&\!\!\!d_{G_{n},k}H_{n+k},\vs{6pt}\\
d_{\0,k}(H_{n})\!\!\!&=&\!\!\!c_{H_{n},k}G_{n+k}\!\!\!&+&\!\!\!d_{H_{n},k}H_{n+k}.
\end{array}
\end{eqnarray}
The action of $d_{\0,k}$ on both sides of $[L_{n},L_{m}]=(n-m)L_{n+m}$ gives
\begin{eqnarray*}
\begin{array}{ccccc}
(n-m)a_{L_{m+n},k}\!\!\!&=&\!\!\!(n+k-m)a_{L_{n},k}\!\!\!&+&\!\!\!(n-m-k)a_{L_{m},k},\vs{6pt}\\
(n-m)b_{L_{m+n},k}\!\!\!&=&\!\!\!(n+k-m)b_{L_{n},k}\!\!\!&+&\!\!\!(n-m-k)b_{L_{m},k}.
\end{array}
\end{eqnarray*}
When $n=0$, we have
\begin{eqnarray}\label{(0,0,k)-LL}
&&a_{L_{m},k}=(1-\frac{m}{k})a_{L_{0},k},\ \ \ \
b_{L_{m},k}=(1-\frac{m}{k})b_{L_{0},k}.
\end{eqnarray}
Consider the action of $d_{\0,k}$ on both sides of $[L_{n},I_{m}]=(n-m)I_{n+m}$, we have
\begin{eqnarray*}
&&(n-m)a_{I_{m+n},k}=(n-m-k)a_{I_{m},k},\vs{6pt}\\
&&(n-m)b_{I_{m+n},k}=(n+k-m)a_{L_{n},k}+(n-m-k)b_{I_{m},k}.
\end{eqnarray*}
Taking $n=0$ and recalling \eqref{(0,0,k)-LL}, we have
\begin{eqnarray}\label{(0,0,k)-LI}
&&a_{I_{m},k}=0,\ \ \ \ b_{I_{m},k}=(1-\frac{m}{k})a_{L_{0},k}.
\end{eqnarray}
Acting $d_{\0,k}$ on both sides of $[L_{n},H_{m}]=(\frac{n}{2}-m)H_{n+m}$, one has
\begin{eqnarray*}
&&(\frac{n}{2}-m)c_{H_{m+n},k}=(\frac{n}{2}-m-k)c_{H_{m},k},\vs{6pt}\\
&&(\frac{n}{2}-m)d_{H_{m+n},k}=(\frac{n+k}{2}-m)a_{L_{n},k}
+\lambda(n+1)c_{H_{m},k}+(\frac{n}{2}-m-k)d_{H_{m},k}.
\end{eqnarray*}
When $n=0$, together with \eqref{(0,0,k)-LL}, one can deduce
\begin{eqnarray}\label{(0,0,k)-LH}
&&c_{H_{m},k}=0,\ \ \ \ d_{H_{m},k}=(\frac{1}{2}-\frac{m}{k})a_{L_{0},k}.
\end{eqnarray}
Acting $d_{\0,k}$ on both sides of $[L_{n},G_{m}]=(\frac{n}{2}-m)G_{n+m}+\lambda(n+1)H_{m+n}$, we have
\begin{eqnarray*}
(\frac{n}{2}-m)c_{G_{m+n},k}+\lambda(n+1)c_{H_{m+n},k}
\!\!\!&=&\!\!\!(\frac{n+k}{2}-m)a_{L_{n},k}+(\frac{n}{2}-m-k)c_{G_{m},k},\vs{6pt}\\
(\frac{n}{2}-m)d_{G_{m+n},k}
+\lambda(n+1)d_{H_{m+n},k}\!\!\!&=&\!\!\!\lambda(n+k+1)a_{L_{n},k}
+(n+k-2m)b_{L_{n},k},\vs{6pt}\\
&&+\lambda(n+1)c_{G_{m},k}+(\frac{n}{2}-m-k)d_{G_{m},k}.
\end{eqnarray*}
Taking $n=0$ and recalling \eqref{(0,0,k)-LL}, \eqref{(0,0,k)-LH}, we have
\begin{eqnarray}\label{(0,0,k)-LG}
&&c_{G_{m},k}=(\frac{1}{2}-\frac{m}{k})a_{L_{0},k},\ \ \ \
d_{G_{m},k}=\lambda(1+\frac{1}{k})a_{L_{0},k}+(1-\frac{2m}{k})b_{L_{0},k}.
\end{eqnarray}

According to the identities from \eqref{(0,0,k)-LL} to \eqref{(0,0,k)-LG}, we can deduce
\begin{eqnarray}\label{(0,0,k)-ok}
\begin{array}{llllll}
&&d_{\0,k}(L_{n})=(1-\frac{n}{k})a_{L_{0},k}L_{n+k}
+(1-\frac{n}{k})b_{L_{0},k}I_{n+k},\vs{6pt}\\
&&d_{\0,k}(I_{n})=(1-\frac{n}{k})a_{L_{0},k}I_{n+k},\ \ \
d_{\0,k}(H_{n})=(\frac{1}{2}-\frac{n}{k})a_{L_{0},k}H_{n+k},\vs{6pt}\\
&&d_{\0,k}(G_{n})=(\frac{1}{2}-\frac{n}{k})a_{L_{0},k}G_{n+k}
+(\lambda(1+\frac{1}{k})a_{L_{0},k}+(1-\frac{2n}{k})b_{L_{0},k})H_{n+k}.
\end{array}
\end{eqnarray}
It is easy to verify that $d_{\0,k}=\ad(\frac{a_{L_{0},k}}{k}L_{k}+\frac{b_{L_{0},k}}{k}I_{k})$, which implies
$\Der_{\0}({\mathfrak{L}^0_\lambda})\cap\Der_{k}({\mathfrak{L}^0_\lambda})
=\Inn_{\0}({\mathfrak{L}^0_\lambda})\cap\Inn_{k}({\mathfrak{L}^0_\lambda})$ for any $k\in\Z^*$.
\end{proof}

Define $\partial_{1}$ and $\partial_{2}$ as follows:
\begin{eqnarray}\label{partial{1},{2}}
\begin{array}{llllll}
&&\partial_{1}(L_{k})=0,\ \ \,\partial_{1}(I_{k})=2I_{k},
\ \ \,\partial_{1}(G_{k})=G_{k},\ \ \,\partial_{1}(H_{k})=3H_{k},\vs{6pt}\\
&&\partial_{2}(L_{k})=\partial_{2}(I_{k})=\partial_{2}(H_{k})=0,
\ \ \,\partial_{2}(G_{k})=H_{k}.
\end{array}
\end{eqnarray}
Then $\partial_{1}\in \Der_{\0}({\mathfrak{L}^0_0})\cap\Der_{0}({\mathfrak{L}^0_0})$ and $\partial_{2}\in \Der_{\0}({\mathfrak{L}^0_\lambda})\cap\Der_{0}({\mathfrak{L}^0_\lambda})$ for any $\lambda\in\C$.

\begin{lemm}\label{lemma(0,0,0)}
$\Der_{\0}({\mathfrak{L}^0_\lambda})\cap\Der_{0}({\mathfrak{L}^0_\lambda})
=\Inn_{\0}({\mathfrak{L}^0_\lambda})\cap\Inn_{0}({\mathfrak{L}^0_\lambda})
\oplus\C\delta_{\lambda,0}\partial_{1}\oplus\C\partial_{2}$.
\end{lemm}
\begin{proof}

For any $d_{\0,0}\in\Der_{\0}({\mathfrak{L}^0_\lambda})\cap\Der_{0}({\mathfrak{L}^0_\lambda})$, we can suppose
\begin{eqnarray}\label{eq(0,0,0)}
\begin{array}{ccccc}
d_{\0,0}(L_{n})\!\!\!&=&\!\!\!a_{L_{n},0}L_{n}
\!\!\!&+&\!\!\!b_{L_{n},0}I_{n},\vs{6pt}\\
d_{\0,0}(I_{n})\!\!\!&=&\!\!\!a_{I_{n},0}L_{n}
\!\!\!&+&\!\!\!b_{I_{n},0}I_{n},\vs{6pt}\\
d_{\0,0}(G_{n})\!\!\!&=&\!\!\!c_{G_{n},0}G_{n}
\!\!\!&+&\!\!\!d_{G_{n},0}H_{n},\vs{6pt}\\
d_{\0,0}(H_{n})\!\!\!&=&\!\!\!c_{H_{n},0}G_{n}
\!\!\!&+&\!\!\!d_{H_{n},0}H_{n}.
\end{array}
\end{eqnarray}
Acting $d_{\0,0}$ on both sides of $[L_{n},L_{m}]=(n-m)L_{n+m}$, we obtain
\begin{eqnarray}\label{(0,0,0)-LL-000}
\begin{array}{ccccc}
(n-m)a_{L_{m+n},0}\!\!\!&=&\!\!\!(n-m)a_{L_{n},0}\!\!\!&+&\!\!\!(n-m)a_{L_{m},0},\vs{6pt}\\
(n-m)b_{L_{m+n},0}\!\!\!&=&\!\!\!(n-m)b_{L_{n},0}\!\!\!&+&\!\!\!(n-m)b_{L_{m},0},
\end{array}
\end{eqnarray}
which implies the following identities
\begin{eqnarray*}
&&a_{L_{m+n},0}=a_{L_{n},0}+a_{L_{m},0},\ \ \ \ b_{L_{m+n},0}=b_{L_{n}}+b_{L_{m},0},\ \ \ m\neq n.
\end{eqnarray*}
Then we can deduce
\begin{eqnarray}\label{(0,0,0)-LL}
&&a_{L_{k},0}=a_{L_{2},0}+(k-2)a_{L_{1},0},\ \ \ \
a_{L_{-k},0}=a_{L_{-2},0}+(k-2)a_{L_{-1},0},\ \ \ k\geq2,
\end{eqnarray}
which together with $a_{L_{-k},0}+a_{L_{k},0}=a_{L_{0},0}$ gives $a_{L_{0},0}=0$. Then
\begin{eqnarray}\label{(0,0,0)-LL-001}
a_{L_{k},0}=-a_{L_{-k},0},
\end{eqnarray}
which together with $a_{L_{1},0}=a_{L_{2},0}+a_{L_{-1},0}$ gives $a_{L_{2},0}=2a_{L_{1},0}$. Combining \eqref{(0,0,0)-LL} and \eqref{(0,0,0)-LL-001}, we know $a_{L_{k},0}=ka_{L_{1},0}$ for $k\in\Z$. Similarly, we can deduce
\begin{eqnarray}\label{(0,0,0)-LL-results}
a_{L_{k},0}=ka_{L_{1},0},\ \ \ b_{L_{k},0}=kb_{L_{1},0},\ \ \forall\,\,k\in\Z.
\end{eqnarray}

Taking the action of $d_{\0,0}$ on both sides of $[L_{n},I_{m}]=(n-m)I_{n+m}$, we have
\begin{eqnarray*}
\begin{array}{ccccc}
(n-m)a_{I_{m+n},0}\!\!\!&=&\!\!\!(n-m)a_{I_{m},0},\vs{6pt}\\
(n-m)b_{I_{m+n},0}\!\!\!&=&\!\!\!(n-m)a_{L_{n},0}\!\!\!&+&\!\!\!(n-m)b_{I_{m},0},
\end{array}
\end{eqnarray*}
from which we can deduce
\begin{eqnarray}\label{(0,0,0)-LI-results}
&&a_{I_{k},0}=a_{I_{0},0},\ \ \ \ b_{I_{k},0}=b_{I_{0},0}+ka_{L_{1},0},\ \ \forall\,\,k\in\Z.
\end{eqnarray}
Using $[L_{n},H_{m}]=(\frac{n}{2}-m)H_{n+m}$ and \eqref{eq-derivation}, we obtain
\begin{eqnarray*}
&&(\frac{n}{2}-m)c_{H_{m+n},0}=(\frac{n}{2}-m)c_{H_{m},0},\vs{6pt}\\
&&(\frac{n}{2}-m)d_{H_{m+n},0}=(\frac{n}{2}-m)a_{L_{n},0}
+\lambda(n+1)c_{H_{m},0}+(\frac{n}{2}-m)d_{H_{m},0},
\end{eqnarray*}
from which we can deduce
\begin{eqnarray}\label{(0,0,0)-LH-results}
&&c_{H_{k},0}=c_{H_{0},0},\ \ \ \lambda c_{H_{0},0}=0,\ \ \
d_{H_{k},0}=ka_{L_{1},0}+d_{H_{0},0},\ \ \forall\,\,k\in\Z.
\end{eqnarray}
Using $[L_{n},G_{m}]=(\frac{n}{2}-m)G_{n+m}+\lambda (n+1)H_{n+m}$ and \eqref{eq-derivation}, we obtain
\begin{eqnarray}\label{(0,0,0)-LG-000}
\begin{array}{ccccccc}
(\frac{n}{2}-m)c_{G_{m+n},0}\!\!\!&+&\!\!\!\lambda(n+1)c_{H_{m+n},0}
\!\!\!&=&\!\!\!\!\!\!\!\!\!
(\frac{n}{2}-m)a_{L_{n},0}\!\!\!&+&\!\!\!(\frac{n}{2}-m)c_{G_{m},0},\vs{6pt}\\
(\frac{n}{2}-m)d_{G_{m+n},0}\!\!\!&+&\!\!\!\lambda (n+1)d_{H_{m+n},0}\!\!\!&=&\!\!\!\!\!\!\!\!\!
\lambda(n+1)a_{L_{n},0}\!\!\!&+&\!\!\!(n-2m)b_{L_{n},0}\vs{6pt}\\
& & & &+\lambda(n+1)c_{G_{m},0}\!\!\!&+&\!\!\!(\frac{n}{2}-m)d_{G_{m},0}.
\end{array}
\end{eqnarray}
Letting $m=0$ in the first equation in \eqref{(0,0,0)-LG-000}, one can obtain
\begin{eqnarray}\label{(0,0,0)-LG-results001}
c_{G_{k},0}=ka_{L_{1},0}+c_{G_{0},0},\ \ \forall\,\,k\in\Z.
\end{eqnarray}
Letting $n=0$ in  the second equation in \eqref{(0,0,0)-LG-000}, one can obtain
\begin{eqnarray}\label{(0,0,0)-LG-results002}
\lambda d_{H_{0},0}=\lambda c_{G_{0},0}.
\end{eqnarray}
Letting $m=0$ in  the second equation in \eqref{(0,0,0)-LG-000}, one can obtain
\begin{eqnarray}\label{(0,0,0)-LG-results003}
d_{G_{k},0}=2kb_{L_{1},0}+d_{G_{0},0},\ \ \forall\,\,k\in\Z.
\end{eqnarray}

Acting $d_{\0,0}$ on both sides of $[I_{n},G_{m}]=(n-2m)H_{n+m}$, we have
\begin{eqnarray*}
&&(n-2m)c_{H_{m+n},0}=(\frac{n}{2}-m)a_{I_{n},0},\vs{6pt}\\
&&(n-2m)d_{H_{m+n},0}=\lambda(n+1)a_{I_{n},0}+(n-2m)b_{I_{n},0}+(n-2m)c_{G_{m},0},
\end{eqnarray*}
which imply
\begin{eqnarray}\label{(0,0,0)-IG-results}
a_{I_{k},0}=a_{I_{0},0}=2c_{H_{0},0},\ \ \ \
d_{H_{0},0}=b_{I_{0},0}+c_{G_{0},0},\ \ \forall\,\,k\in\Z.
\end{eqnarray}
Using $[G_{n},G_{m}]=I_{m+n}$ and \eqref{eq-derivation}, we obtain
\begin{eqnarray*}
&&a_{I_{m+n},0}=0,\ \ \ \ b_{I_{m+n},0}=c_{G_{m},0}+c_{G_{n},0},
\end{eqnarray*}
which gives
\begin{eqnarray}\label{(0,0,0)-GG-results}
a_{I_{k},0}=0,\ \ \ \ b_{I_{0},0}=2c_{G_{0},0},\ \ \forall\,\,k\in\Z.
\end{eqnarray}

Combining the identities obtained in \eqref{(0,0,0)-LL-results}--\eqref{(0,0,0)-GG-results}, we have deduced the following
\begin{eqnarray*}
d_{\0,0}(L_{k})\!\!\!&=&\!\!\!ka_{L_{1},0}L_{k}+kb_{L_{1},0}I_{k},\vs{6pt}\\
d_{\0,0}(G_{k})\!\!\!&=&\!\!\!(ka_{L_{1},0}+c_{G_{0},0})G_{k}
+(2kb_{L_{1},0}+d_{G_{0},0})H_{k},\vs{6pt}\\
d_{\0,0}(I_{k})\!\!\!&=&\!\!\!(ka_{L_{1},0}+2c_{G_{0},0})I_{k},\ \ \ \
d_{\0,0}(H_{k})=(ka_{L_{1},0}+3c_{G_{0},0})G_{k}.
\end{eqnarray*}
It is easy to see that $d_{\0,0}=\ad(-a_{L_{1},0}L_{0}-b_{L_{1},0}I_{0})+ d_{G_{0},0}\partial_{2}+\delta_{\lambda,0}c_{G_{0},0}\partial_{1}$, which implies
$\Der_{\0}({\mathfrak{L}^0_\lambda})\cap\Der_{0}({\mathfrak{L}^0_\lambda})=
\Inn_{\0}({\mathfrak{L}^0_\lambda})\cap\Inn_{0}({\mathfrak{L}^0_\lambda})
\oplus\C\delta_{\lambda,0}\partial_{1}\oplus\C\partial_{2}$.
\end{proof}

\begin{lemm}\label{lemma(0,1,k)}
$\Der_{\1}({\mathfrak{L}^0_\lambda})\cap\Der_{k}({\mathfrak{L}^0_\lambda})
=\Inn_{\1}({\mathfrak{L}^0_\lambda})\cap\Inn_{k}({\mathfrak{L}^0_\lambda}),\ \ \forall\,\,k\in\Z^*$.
\end{lemm}
\begin{proof}
For any $d_{\1,k}\in\Der_{\1}({\mathfrak{L}^0_\lambda})\cap\Der_{k}({\mathfrak{L}^0_\lambda})$, one can suppose
\begin{eqnarray}\label{eq(0,1,k)}
\begin{array}{ccccc}
d_{\1,k}(L_{n})\!\!\!&=&\!\!\!c_{L_{n},k}G_{n+k}\!\!\!&+&\!\!\!d_{L_{n},k}H_{n+k},\vs{6pt}\\
d_{\1,k}(I_{n})\!\!\!&=&\!\!\!c_{I_{n},k}G_{n+k}\!\!\!&+&\!\!\!d_{I_{n},k}H_{n+k},\vs{6pt}\\
d_{\1,k}(G_{n})\!\!\!&=&\!\!\!a_{G_{n},k}L_{n+k}\!\!\!&+&\!\!\!b_{G_{n},k}I_{n+k},\vs{6pt}\\
d_{\1,k}(H_{n})\!\!\!&=&\!\!\!a_{H_{n},k}L_{n+k}\!\!\!&+&\!\!\!b_{H_{n},k}I_{n+k}.
\end{array}
\end{eqnarray}
Using $[L_{n},L_{m}]=(n-m)L_{m+n}$ and \eqref{eq-derivation}, we obtain
\begin{eqnarray*}
\begin{array}{ccccccc}
(n-m)c_{L_{m+n},k}\!\!\!&=&\!\!\!(\frac{n}{2}-m-k)c_{L_{m},k}
-(\frac{m}{2}-n-k)c_{L_{n},k},\vs{6pt}\\
(n-m)d_{L_{m+n},k}\!\!\!&=&\!\!\!\!\!\!\!\!\!
\lambda(n+1)c_{L_{m},k}+(\frac{n}{2}-m-k)d_{L_{m},k}\vs{6pt}\\
&&-\lambda(m+1)c_{L_{n},k}-(\frac{m}{2}-n-k)d_{L_{n},k}.
\end{array}
\end{eqnarray*}
Taking $n=0$, we obtain
\begin{eqnarray}\label{(0,1,k)-LL-results}
\begin{array}{lll}
&&c_{L_{m},k}=(1-\frac{m}{2k})c_{L_{0},k},\vs{6pt}\\
&&d_{L_{m},k}=(1-\frac{m}{2k})d_{L_{0},k}
-(\frac{1}{2k^2}+\frac{1}{k})\lambda mc_{L_{0},k}.
\end{array}
\end{eqnarray}
Using $[L_{n},I_{m}]=(n-m)I_{m+n}$ and \eqref{eq-derivation}, we get
\begin{eqnarray*}
&&(n-m)c_{I_{m+n},k}=(\frac{n}{2}-m-k)c_{I_{m},k},\vs{6pt}\\
&&(n-m)d_{I_{m+n},k}=\lambda(n+1)c_{I_{m},k}
+(\frac{n}{2}-m-k)d_{I_{m},k}-(m-2n-2k)c_{L_{n},k}.
\end{eqnarray*}
Taking $n=0$ in \eqref{(0,1,k)-LL-results}, we obtain
\begin{eqnarray}\label{(0,1,k)-LI-results}
&&c_{I_{m},k}=0,\ \ \ d_{I_{m},k}=(2-\frac{m}{k})c_{L_{0},k}.
\end{eqnarray}
Acting $d_{\1,k}$ on both sides of $[L_{n},H_{m}]=(\frac{n}{2}-m)H_{m+n}$, we have
\begin{eqnarray*}
\begin{array}{ccc}
(\frac{n}{2}-m)a_{H_{m+n},k}\!\!\!&=&\!\!\!(n-m-k)a_{H_{m},k},\vs{6pt}\\
(\frac{n}{2}-m)b_{H_{m+n},k}\!\!\!&=&\!\!\!(n-m-k)b_{H_{m},k}.
\end{array}
\end{eqnarray*}
Taking $n=0$, we obtain
\begin{eqnarray}\label{(0,1,k)-LH-results}
&&a_{H_{m},k}=0,\ \ \ b_{H_{m},k}=0.
\end{eqnarray}
Using $[L_{n},G_{m}]=(\frac{n}{2}-m)G_{m+n}+\lambda(n+1)H_{m+n}$, we obtain
\begin{eqnarray*}
\begin{array}{ccccccc}
(\frac{n}{2}-m)a_{G_{m+n},k}\!\!\!&+&\!\!\!\lambda(n+1)a_{H_{m+n},k}
\!\!\!&=&\!\!\!(n-m-k)a_{G_{m},k},\vs{6pt}\\
(\frac{n}{2}-m)b_{G_{m+n},k}\!\!\!&+&\!\!\!\lambda(n+1)b_{H_{m+n},k}
\!\!\!&=&\!\!\!(n-m-k)b_{G_{m},k}\!\!\!&+&\!\!\!c_{L_{n},k}.
\end{array}
\end{eqnarray*}
Taking $n=0$ and combining \eqref{(0,1,k)-LH-results} with \eqref{(0,1,k)-LL-results}, we obtian
\begin{eqnarray}\label{(0,1,k)-LG-results}
&&a_{G_{m},k}=0,\ \ \ b_{G_{m},k}=\frac{1}{k}c_{L_{0},k}.
\end{eqnarray}
Recalling \eqref{(0,1,k)-LL-results}--\eqref{(0,1,k)-LG-results}, one can deduce
\begin{eqnarray}\label{(0,1,k)-results}
\begin{array}{ccccc}
d_{\1,k}(L_{m})\!\!\!&=&\!\!\!(1-\frac{m}{2k})c_{L_{0},k}G_{m+k}
+((1-\frac{m}{2k})d_{L_{0},k}-(\frac{1}{2k^2}+\frac{1}{k})\lambda mc_{L_{0},k})H_{m+k},\vs{6pt}\\
d_{\1,k}(I_{m})\!\!\!&=&\!\!\!(2-\frac{m}{k})c_{L_{0},k}H_{m+k},\ \ \
d_{\1,k}(G_{m})=\frac{1}{k}c_{L_{0},k}I_{m+k},\ \ \ d_{\1,k}(H_{m})=0.
\end{array}
\end{eqnarray}
It is easy to verify that
$d_{\1,k}=\ad (\frac{c_{L_{0},k}}{k}G_{k}+(\frac{d_{L_{0},k}}{k}
+\frac{\lambda}{k^2}c_{L_{0},k})H_{k})$, which implies $\Der_{\1}({\mathfrak{L}^0_\lambda})\cap\Der_{k}({\mathfrak{L}^0_\lambda})
=\Inn_{\1}({\mathfrak{L}^0_\lambda})\cap\Inn_{k}({\mathfrak{L}^0_\lambda})$ for any $k\in\Z^*$.
\end{proof}

\begin{lemm}\label{lemma(0,1,0)}
$\Der_{\1}({\mathfrak{L}^0_\lambda})\cap\Der_{0}({\mathfrak{L}^0_\lambda})
=\Inn_{\1}({\mathfrak{L}^0_\lambda})\cap\Inn_{0}({\mathfrak{L}^0_\lambda})$.
\end{lemm}
\begin{proof}
For any $d_{\1,0}\in\Der_{\1}({\mathfrak{L}^0_\lambda})
\cap\Der_{0}({\mathfrak{L}^0_\lambda})$, we can suppose
\begin{eqnarray}\label{eq(0,1,0)}
\begin{array}{ccccc}
d_{\1,0}(L_{n})\!\!\!&=&\!\!\!c_{L_{n},0}G_{n}
\!\!\!&+&\!\!\!d_{L_{n},0}H_{n},\vs{6pt}\\
d_{\1,0}(I_{n})\!\!\!&=&\!\!\!c_{I_{n},0}G_{n}
\!\!\!&+&\!\!\!d_{I_{n},0}H_{n},\vs{6pt}\\
d_{\1,0}(G_{n})\!\!\!&=&\!\!\!a_{G_{n},0}L_{n}
\!\!\!&+&\!\!\!b_{G_{n},0}I_{n},\vs{6pt}\\
d_{\1,0}(H_{n})\!\!\!&=&\!\!\!a_{H_{n},0}L_{n}\!\!\!&+&\!\!\!b_{H_{n},0}I_{n}.
\end{array}
\end{eqnarray}
Using $[L_{n},L_{m}]=(n-m)L_{n+m}$ and \eqref{eq-derivation}, we obtain
\begin{eqnarray}\label{(0,1,0)-LL-000}
\begin{array}{lll}
&&\!\!\!\!\!\!\!\!\!
(n-m)c_{L_{m+n},0}=(\frac{n}{2}-m)c_{L_{m},0}-(\frac{m}{2}-n)c_{L_{n},0},\vs{6pt}\\
&&\!\!\!\!\!\!\!\!\!
(n-m)d_{L_{m+n},0}=\lambda(n+1)c_{L_{m},0}+(\frac{n}{2}-m)d_{L_{m},0}
-\lambda(m+1)c_{L_{n},0}-(\frac{m}{2}-n)d_{L_{n},0},
\end{array}
\end{eqnarray}
which give $c_{L_{0},0}=0$ and
\begin{eqnarray}\label{(0,1,0)-LL-results-001}
\lambda c_{L_{m},0}=\frac{m}{2}d_{L_{0},0}.
\end{eqnarray}
Then the second equation in \eqref{(0,1,0)-LL-000} becomes
\begin{equation*}
(n-m)d_{L_{m+n},0}=\frac{m-n}{2}d_{L_{0},0}+(\frac{n}{2}-m)d_{L_{m},0}-(\frac{m}{2}-n)d_{L_{n},0}.
\end{equation*}
Furthermore, we can deduce
\begin{eqnarray}\label{(0,1,0)-LL-results-002}
&&c_{L_{m},0}=mc_{L_{1},0},\ \ \ d_{L_{m},0}=md_{L_{1},0}-(m-1)d_{L_{0},0}.
\end{eqnarray}
Using $[L_{n},I_{m}]=(n-m)I_{n+m}$ and \eqref{eq-derivation}, we obtain
\begin{eqnarray*}
&&(n-m)c_{I_{m+n},0}=(\frac{n}{2}-m)c_{I_{m},0},\vs{6pt}\\
&&(n-m)d_{I_{m+n},0}=\lambda(n+1)c_{I_{m},0}
+(\frac{n}{2}-m)d_{I_{m},0}-(m-2n)c_{L_{n},0},
\end{eqnarray*}
from which we can deduce
\begin{eqnarray}\label{(0,1,0)-LI-results-001}
c_{I_{n},0}=\frac{1}{2}c_{I_{0},0},\ \ \lambda c_{I_{k},0}=0,\ \
d_{I_{m},0}=\frac{1}{2}d_{I_{0},0}+2mc_{L_{1},0},\ \
\forall\,\,n\in\Z^*,\ \ \,\,k,\,m\in\Z.
\end{eqnarray}
Using $[L_{n},H_{m}]=(\frac{n}{2}-m)H_{m+n}$ and \eqref{eq-derivation}, we have
\begin{eqnarray*}
\begin{array}{ccc}
(\frac{n}{2}-m)a_{H_{m+n},0}\!\!\!&=&\!\!\!(n-m)a_{H_{m},0},\vs{6pt}\\
(\frac{n}{2}-m)b_{H_{m+n},0}\!\!\!&=&\!\!\!(n-m)b_{H_{m},0},
\end{array}
\end{eqnarray*}
which imply
\begin{eqnarray}\label{(0,1,0)-LH-results}
&&a_{H_{k},0}=2a_{H_{0},0},\ \ \ b_{H_{k},0}=2b_{H_{0},0},\ \ \forall\,\,k\in\Z^*.
\end{eqnarray}
Using $[L_{n},G_{m}]=(\frac{n}{2}-m)G_{m+n}+\lambda(n+1)H_{m+n}$ and \eqref{eq-derivation}, we obtain
\begin{eqnarray*}
\begin{array}{ccccccc}
(\frac{n}{2}-m)a_{G_{m+n},0}\!\!\!&+&\!\!\!\lambda(n+1)a_{H_{m+n},0}
\!\!\!&=&\!\!\!(n-m)a_{G_{m},0},\vs{6pt}\\
(\frac{n}{2}-m)b_{G_{m+n},0}\!\!\!&+&\!\!\!\lambda(n+1)b_{H_{m+n},0}
\!\!\!&=&\!\!\!(n-m)b_{G_{m},0}\!\!\!&+&\!\!\!c_{L_{n},0},
\end{array}
\end{eqnarray*}
which imply
\begin{eqnarray}\label{(0,1,0)-LG-results}
\lambda a_{H_{k},0}=\lambda b_{H_{k},0}=0,\ \,
a_{G_{n},0}=2a_{G_{0},0},\ \,
b_{G_{n},0}=2b_{G_{0},0}+\frac{2}{n}c_{L_{n},0},
\ \,\forall\,\,k\in\Z,\,\,n\in\Z^*.
\end{eqnarray}
Using $[I_{n},G_{m}]=(n-2m)H_{m+n}$ and \eqref{eq-derivation}, we have
\begin{eqnarray*}
&&(n-2m)a_{H_{m+n},0}=0,\ \ \ (n-2m)b_{H_{m+n},0}=(n-m)a_{G_{m},0}+c_{I_{n},0},
\end{eqnarray*}
which imply
\begin{eqnarray}\label{(0,1,0)-IG-results}
&&a_{H_{k},0}=a_{G_{k},0}=c_{I_{k},0}=0,\ \,\forall\,\,k\in\Z.
\end{eqnarray}
Using $[G_{n},G_{m}]=I_{m+n}$ and \eqref{eq-derivation}, we have
\begin{eqnarray*}
\begin{array}{ccccc}
c_{I_{m+n},0}\!\!\!&=&\!\!\!\!\!\!\!\!\!
(\frac{n}{2}-m)a_{G_{n},0}+(\frac{m}{2}-n)a_{G_{m},0},\vs{6pt}\\
d_{I_{m+n},0}\!\!\!&=&\!\!\!\!\!\!\!\!\!
\lambda(n+1)a_{G_{n},0}+(n-2m)b_{G_{n},0}\vs{6pt}\\
&&+\lambda(m+1)a_{G_{m},0}+(m-2n)b_{G_{m},0}.
\end{array}
\end{eqnarray*}
Together with \eqref{(0,1,0)-IG-results}, we can deduce
\begin{equation*}
d_{I_{0},0}=-4(m+n)(b_{G_{0},0}+2c_{L_{1},0}),
\end{equation*}
which implies
\begin{eqnarray}\label{(0,1,0)-GG-results}
&&d_{I_{0},0}=0,\ \ \ b_{G_{0},0}=-2c_{L_{1},0}.
\end{eqnarray}

Combining the identities presented in \eqref{(0,1,0)-LL-results-001}--\eqref{(0,1,0)-GG-results}, we can deduce
\begin{eqnarray*}
&&d_{\1,0}(L_{k})=kc_{L_{1},0}G_{k}+(kd_{L_{1},0}-2\lambda (k-1)c_{L_{1},0})H_{k},\vs{6pt}\\
&&d_{\1,0}(I_{k})=2kc_{L_{1},0}H_{k},\ \ \
d_{\1,0}(G_{k})=-2c_{L_{1},0}I_{k},\ \ \ d_{\1,0}(H_{k})=0.
\end{eqnarray*}
It is easy to verify that $d_{\1,0}=\ad(-2c_{L_{1},0}G_{0}+(8\lambda c_{L_{1},0}-2d_{L_{1},0})H_{0})$, which imply
$\Der_{\1}({\mathfrak{L}^0_\lambda})
\cap\Der_{0}({\mathfrak{L}^0_\lambda})
=\Inn_{\1}({\mathfrak{L}^0_\lambda})\cap\Inn_{0}({\mathfrak{L}^0_\lambda})$.
\end{proof}

\begin{lemm}\label{lemma(0.5,0,k)}
$\Der_{\0}({\mathfrak{L}^{\frac{1}{2}}_\lambda})
\cap\Der_{k}({\mathfrak{L}^{\frac{1}{2}}_\lambda})=\Inn_{\0}({\mathfrak{L}^{\frac{1}{2}}_\lambda})\cap\Inn_{k}({\mathfrak{L}^{\frac{1}{2}}_\lambda})$ for any $k\in\Z^*$.
\end{lemm}
\begin{proof}
For any $d_{\0,k}\in\Der_{\0}({\mathfrak{L}^{\frac{1}{2}}_\lambda})
\cap\Der_{k}({\mathfrak{L}^{\frac{1}{2}}_\lambda})$, we can suppose
\begin{eqnarray}\label{eq(0.5,0,k)}
\begin{array}{llllll}
&&d_{\0,k}(L_{n})=a_{L_{n},k}L_{n+k}
+b_{L_{n},k}I_{n+k},\vs{6pt}\\
&&d_{\0,k}(I_{n})=a_{I_{n},k}L_{n+k}
+b_{I_{n},k}I_{n+k},\vs{6pt}\\
&&d_{\0,k}(G_{n+\frac{1}{2}})
=c_{G_{n+\frac{1}{2}},k}G_{n+k+\frac{1}{2}}
+d_{G_{n+\frac{1}{2}},k}H_{n+k+\frac{1}{2}},\vs{6pt}\\
&&d_{\0,k}(H_{n+\frac{1}{2}})
=c_{H_{n+\frac{1}{2}},k}G_{n+k+\frac{1}{2}}
+d_{H_{n+\frac{1}{2}},k}H_{n+k+\frac{1}{2}}.
\end{array}
\end{eqnarray}
Acting $d_{\0,k}$ on both sides of $[L_{n},L_{m}]=(n-m)L_{n+m}$, we obtain
\begin{eqnarray*}
&&(n-m)a_{L_{m+n},k}=(n+k-m)a_{L_{n},k}+(n-m-k)a_{L_{m},k},\vs{6pt}\\
&&(n-m)b_{L_{m+n},k}=(n+k-m)b_{L_{n},k}+(n-m-k)b_{L_{m},k}.
\end{eqnarray*}
Taking $n=0$, one has
\begin{eqnarray}\label{(0.5,0,k)-LL}
&&a_{L_{m},k}=(1-\frac{m}{k})a_{L_{0},k},\ \ \ b_{L_{m},k}=(1-\frac{m}{k})b_{L_{0},k}.
\end{eqnarray}
Using $[L_{n},I_{m}]=(n-m)I_{n+m}$ and \eqref{eq(0,0,k)}, we obtain
\begin{eqnarray*}
&&(n-m)a_{I_{m+n},k}=(n-m-k)a_{I_{m},k},\vs{6pt}\\
&&(n-m)b_{I_{m+n},k}=(n+k-m)a_{L_{n},k}+(n-m-k)b_{I_{m},k},
\end{eqnarray*}
which give
\begin{eqnarray*}
&&a_{I_{m},k}=0,\ \ \ b_{I_{m},k}=(1-\frac{m}{k})a_{L_{0},k}.
\end{eqnarray*}
Using $[L_{n},H_{m+\frac{1}{2}}]=(\frac{n-1}{2}-m)H_{n+m+\frac{1}{2}}$ and \eqref{eq(0.5,0,k)}, we obtain
\begin{eqnarray*}
(\frac{n-1}{2}-m)c_{H_{m+n+\frac{1}{2}},k}
\!\!\!&=&\!\!\!(\frac{n-1}{2}-m-k)c_{H_{m+\frac{1}{2}},k},\vs{6pt}\\
(\frac{n-1}{2}-m)d_{H_{m+n+\frac{1}{2}},k}
\!\!\!&=&\!\!\!(\frac{n+k-1}{2}-m)a_{L_{n},k}+\lambda(n+1)c_{H_{m+\frac{1}{2}},k}\vs{6pt}\\
&&+(\frac{n-1}{2}-m-k)d_{H_{m+\frac{1}{2}},k},
\end{eqnarray*}
which together with \eqref{(0.5,0,k)-LL} give
\begin{eqnarray}\label{(0.5,0,k)-LH}
&&c_{H_{m+\frac{1}{2}},k}=0,\ \ \
d_{H_{m+\frac{1}{2}},k}=(\frac{1}{2}-\frac{1}{2k}-\frac{m}{k})a_{L_{0},k}.
\end{eqnarray}
Using $[L_{n},G_{m+\frac{1}{2}}]=(\frac{n-1}{2}-m)G_{n+m+\frac{1}{2}}
+\lambda(n+1)H_{m+n+\frac{1}{2}}$ and \eqref{eq(0.5,0,k)}, we obtain
\begin{eqnarray*}
\begin{array}{ccccccc}
(\frac{n-1}{2}-m)c_{G_{m+n+\frac{1}{2}},k}
\!\!\!&+&\!\!\!\lambda(n+1)c_{H_{m+n+\frac{1}{2}},k}
\!\!\!&=&\!\!\!(\frac{n+k-1}{2}-m)a_{L_{n},k}
\!\!\!&+&\!\!\!(\frac{n-1}{2}-m-k)c_{G_{m+\frac{1}{2}},k},\vs{6pt}\\
(\frac{n-1}{2}-m)d_{G_{m+n+\frac{1}{2}},k}
\!\!\!&+&\!\!\!\lambda(n+1)d_{H_{m+n+\frac{1}{2}},k}
\!\!\!&=&\!\!\!\lambda(n+k+1)a_{L_{n},k}\!\!\!&+&\!\!\!(n+k-2m-1)b_{L_{n},k}\\
& & & & +\lambda(n+1)c_{G_{m+\frac{1}{2}},k}\!\!\!&+&\!\!\!
(\frac{n-1}{2}-m-k)d_{G_{m+\frac{1}{2}},k},
\end{array}
\end{eqnarray*}
which together with \eqref{(0.5,0,k)-LL} and \eqref{(0.5,0,k)-LH} give
\begin{eqnarray}\label{(0.5,0,k)-LG}
\begin{array}{lll}
&&c_{G_{m+\frac{1}{2}},k}=(\frac{1}{2}-\frac{1}{2k}-\frac{m}{k})a_{L_{0},k},\vs{6pt}\\
&&d_{G_{m+\frac{1}{2}},k}=\lambda(1+\frac{1}{k})a_{L_{0},k}
+(1-\frac{2m}{k}-\frac{1}{k})b_{L_{0},k}.
\end{array}
\end{eqnarray}
Combining the identities presented in \eqref{(0.5,0,k)-LL}--\eqref{(0.5,0,k)-LG}, we obtain
\begin{eqnarray*}
&&d_{\0,k}(L_{n})=(1-\frac{n}{k})a_{L_{0},k}L_{n+k}
+(1-\frac{n}{k})b_{L_{0},k}I_{n+k},\vs{6pt}\\
&&d_{\0,k}(I_{n})=(1-\frac{n}{k})a_{L_{0},k}I_{n+k},\ \ \
d_{\0,k}(H_{n})=(\frac{1}{2}-\frac{n}{k})a_{L_{0},k}H_{n+k},\vs{6pt}\\
&&d_{\0,k}(G_{n+\frac{1}{2}})=(\frac{1}{2}-\frac{1}{2k}
-\frac{n}{k})a_{L_{0},k}G_{n+k+\frac{1}{2}}+
(\lambda(1+\frac{1}{k})a_{L_{0},k}+(1-\frac{1}{k}-\frac{2n}{k})b_{L_{0},k})
H_{n+k+\frac{1}{2}}.
\end{eqnarray*}
It is easy to verify that $d_{\0,k}=\ad(\frac{a_{L_{0},k}}{k}L_{k}+\frac{b_{L_{0},k}}{k}I_{k})$, which implies $\Der_{\0}({\mathfrak{L}^{\frac{1}{2}}_\lambda})
\cap\Der_{k}({\mathfrak{L}^{\frac{1}{2}}_\lambda})
=\Inn_{\0}({\mathfrak{L}^{\frac{1}{2}}_\lambda})
\cap\Inn_{k}({\mathfrak{L}^{\frac{1}{2}}_\lambda})$ for any $k\in\Z^*$.
\end{proof}

Define $\partial_{3}$ and $\partial_{4}$ as follows:
\begin{eqnarray}\label{partial34}
\begin{array}{lll}
&&\partial_{3}(L_{k})=0, \partial_{3}(I_{k})=2I_{k}, \partial_{3}(G_{k+\frac{1}{2}})=G_{k+\frac{1}{2}}, \partial_{3}(H_{k+\frac{1}{2}})=3H_{k+\frac{1}{2}},\vs{6pt}\\
&&\partial_{4}(L_{k})=0, \partial_{4}(I_{k})=0, \partial_{4}(G_{k+\frac{1}{2}})=H_{k+\frac{1}{2}}, \partial_{4}(H_{k+\frac{1}{2}})=0.
\end{array}
\end{eqnarray}
It is easy to see that $\partial_{3}\in\Der_{\0}({\mathfrak{L}^{\frac{1}{2}}_0})
\cap\Der_{0}({\mathfrak{L}^{\frac{1}{2}}_0})$ and $\partial_{4}\in \Der_{\0}({\mathfrak{L}^{\frac{1}{2}}_\lambda})
\cap\Der_{0}({\mathfrak{L}^{\frac{1}{2}}_\lambda})$ for any $\lambda\in\C$.

\begin{lemm}\label{lemma(0.5,0,0)}
$\Der_{\0}({\mathfrak{L}^{\frac{1}{2}}_\lambda})
\cap\Der_{0}({\mathfrak{L}^{\frac{1}{2}}_\lambda})=
\Inn_{\0}({\mathfrak{L}^{\frac{1}{2}}_\lambda})
\cap\Inn_{0}({\mathfrak{L}^{\frac{1}{2}}_\lambda})
\oplus\C\delta_{\lambda,0}\partial_{3}\oplus\C\partial_{4}$.
\end{lemm}
\begin{proof}
For any $d_{\0,0}\in\Der_{\0}({\mathfrak{L}^{\frac{1}{2}}_\lambda})\cap\Der_{0}({\mathfrak{L}^{\frac{1}{2}}_\lambda})$, we can suppose
\begin{eqnarray*}
&&d_{\0,0}(L_{n})=a_{L_{n},0}L_{n}+b_{L_{n},0}I_{n},\vs{6pt}\\
&&d_{\0,0}(I_{n})=a_{I_{n},0}L_{n}+b_{I_{n},0}I_{n},\vs{6pt}\\
&&d_{\0,0}(G_{n+\frac{1}{2}})=c_{G_{n+\frac{1}{2}},0}G_{n+\frac{1}{2}}
+d_{G_{n+\frac{1}{2}},0}H_{n+\frac{1}{2}},\vs{6pt}\\
&&d_{\0,0}(H_{n+\frac{1}{2}})=c_{H_{n+\frac{1}{2}},0}G_{n+\frac{1}{2}}
+d_{H_{n+\frac{1}{2}},0}H_{n+\frac{1}{2}}.
\end{eqnarray*}
Using $[L_{n},L_{m}]=(n-m)L_{n+m}$ and \eqref{eq-derivation}, we obtain
\begin{eqnarray}\label{(0.5,0,0)-LL-000}
\begin{array}{ccccc}
(n-m)a_{L_{m+n},0}\!\!\!&=&\!\!\!(n-m)a_{L_{n},0}
\!\!\!&+&\!\!\!(n-m)a_{L_{m},0},\vs{6pt}\\
(n-m)b_{L_{m+n},0}\!\!\!&=&\!\!\!(n-m)b_{L_{n},0}
\!\!\!&+&\!\!\!(n-m)b_{L_{m},0},
\end{array}
\end{eqnarray}
which imply
\begin{eqnarray*}
&&a_{L_{m+n},0}=a_{L_{n},0}+a_{L_{m},0},\ \ \
b_{L_{m+n},0}=b_{L_{n}}+b_{L_{m},0},\ \ m\neq n.
\end{eqnarray*}
Then we can deduce
\begin{eqnarray}\label{(0.5,0,0)-LL}
\begin{array}{lll}
&&a_{L_{k},0}=a_{L_{2},0}+(k-2)a_{L_{1},0},\vs{6pt}\\
&&a_{L_{-k},0}=a_{L_{-2},0}+(k-2)a_{L_{-1},0},\ \ \forall\,\,k\geq2,
\end{array}
\end{eqnarray}
which together with $a_{L_{-k},0}+a_{L_{k},0}=a_{L_{0},0}$ give $a_{L_{0},0}=0$ and \begin{eqnarray}\label{(0.5,0,0)-LL-001}
a_{L_{k},0}=-a_{L_{-k},0}.
\end{eqnarray}
According to $a_{L_{1},0}=a_{L_{2},0}+a_{L_{-1},0}$ and \eqref{(0.5,0,0)-LL-001}, we know $a_{L_{2},0}=2a_{L_{1},0}$. Then $a_{L_{k},0}=ka_{L_{1},0}$ for $k\in\Z$. Similarly, we can deduce
\begin{eqnarray}\label{(0.5,0,0)-LL-results}
&&a_{L_{k},0}=ka_{L_{1},0},\ \ \ b_{L_{k},0}=kb_{L_{1},0},\ \ \forall\,\,k\in\Z.
\end{eqnarray}
Using $[L_{n},I_{m}]=(n-m)I_{n+m}$ and \eqref{eq-derivation}, we obtain
\begin{eqnarray*}
\begin{array}{ccccc}
(n-m)a_{I_{m+n},0}\!\!\!&=&\!\!\!(n-m)a_{I_{m},0},\vs{6pt}\\
(n-m)b_{I_{m+n},0}\!\!\!&=&\!\!\!(n-m)a_{L_{n},0}\!\!\!&+&\!\!\!(n-m)b_{I_{m},0},
\end{array}
\end{eqnarray*}
which together with \eqref{(0.5,0,0)-LL-results} give
\begin{eqnarray}\label{(0.5,0,0)-LI-results}
&&a_{I_{k},0}=a_{I_{0},0},\ \ \ b_{I_{k},0}=b_{I_{0},0}+ka_{L_{1},0},
\ \ \forall\,\,k\in\Z.
\end{eqnarray}
Using $[L_{n},H_{m+\frac{1}{2}}]=(\frac{n-1}{2}-m)H_{n+m+\frac{1}{2}}$ and \eqref{eq-derivation}, we have
\begin{eqnarray*}
&&\!\!\!\!\!\!\!\!\!
(\frac{n-1}{2}-m)c_{H_{m+n+\frac{1}{2}},0}
=(\frac{n-1}{2}-m)c_{H_{m+\frac{1}{2}},0},\vs{6pt}\\
&&\!\!\!\!\!\!\!\!\!
(\frac{n-1}{2}-m)d_{H_{m+n+\frac{1}{2}},0}
=(\frac{n-1}{2}-m)a_{L_{n},0}+\lambda(n+1)c_{H_{m+\frac{1}{2},0}}
+(\frac{n-1}{2}-m)d_{H_{m+\frac{1}{2}},0},
\end{eqnarray*}
which together with \eqref{(0.5,0,0)-LL-results} give
\begin{eqnarray}\label{(0.5,0,0)-LH-results}
&&c_{H_{k+\frac{1}{2}},0}=c_{H_{\frac{1}{2}},0},\ \ \
\lambda c_{H_{\frac{1}{2}},0}=0,\ \ \
d_{H_{k+\frac{1}{2}},0}=ka_{L_{1},0}+d_{H_{\frac{1}{2}},0},\ \ \forall\,\,k\in\Z.
\end{eqnarray}
Using $[L_{n},G_{m+\frac{1}{2}}]=(\frac{n-1}{2}-m)G_{n+m+\frac{1}{2}}+\lambda (n+1)H_{n+m+\frac{1}{2}}$ and \eqref{eq-derivation}, we obtain
\begin{eqnarray}\label{(0.5,0,0)-LG-000}
\begin{array}{ccccccc}
(\frac{n-1}{2}-m)c_{G_{m+n+\frac{1}{2}},0}
\!+\!\lambda(n+1)c_{H_{m+n+\frac{1}{2}},0}
\!\!\!&=&\!\!\!\!\!\!\!\!\!
(\frac{n-1}{2}-m)a_{L_{n},0}
\!+\!(\frac{n-1}{2}-m)c_{G_{m+\frac{1}{2}},0},\vs{6pt}\\
(\frac{n-1}{2}-m)d_{G_{m+n+\frac{1}{2}},0}+\lambda (n+1)d_{H_{m+n+\frac{1}{2}},0}\!\!\!&=&\!\!\!\!\!\!\!\!\!\!\!\!
\lambda(n+1)a_{L_{n},0}\!+\!(n-2m-1)b_{L_{n},0},\vs{6pt}\\
&&\!+\!\lambda(n+1)c_{G_{m+\frac{1}{2}},0}
\!+\!(\frac{n-1}{2}-m)d_{G_{m+\frac{1}{2}},0},
\end{array}
\end{eqnarray}
from which we can deduce
\begin{eqnarray}\label{(0.5,0,0)-LG-results001}
\begin{array}{lll}
&&c_{G_{k+\frac{1}{2}},0}=ka_{L_{1},0}+c_{G_{\frac{1}{2}},0},\vs{6pt}\\
&&\lambda d_{H_{m+\frac{1}{2}},0}=\lambda c_{G_{m+\frac{1}{2}},0},\vs{6pt}\\
&&d_{G_{n+\frac{1}{2}},0}=2nb_{L_{1},0}+d_{G_{\frac{1}{2}},0},\ \ \forall\,\,k,\,m,\,n\in\Z.
\end{array}
\end{eqnarray}
Acting $d_{\0,0}$ on both sides of $[I_{n},G_{m+\frac{1}{2}}]=(n-2m-1)H_{n+m+\frac{1}{2}}$, we obtain
\begin{eqnarray}\label{(0.5,0,0)-IG}
\begin{array}{lll}
&&(n-2m-1)c_{H_{m+n+\frac{1}{2}},0}=(\frac{n-1}{2}-m)a_{I_{n},0},\vs{6pt}\\
&&(n-2m-1)d_{H_{m+n+\frac{1}{2}},0}=\lambda(n+1)a_{I_{n},0}
+(n-2m-1)(b_{I_{n},0}+c_{G_{m+\frac{1}{2}},0}),
\end{array}
\end{eqnarray}
from which we can deduce
\begin{eqnarray}\label{(0.5,0,0)-IG-results}
&&c_{H_{k+\frac{1}{2}},0}=\frac{a_{I_{0},0}}{2},\ \ \
d_{H_{\frac{1}{2}},0}=\frac{3b_{I_{0},0}+a_{L_{1},0}}{2},\ \ \forall\,\,k\in\Z.
\end{eqnarray}
Using $[G_{n+\frac{1}{2}},G_{m+\frac{1}{2}}]=I_{m+n+1}$ and \eqref{eq-derivation}, we obtain
\begin{eqnarray*}
&&a_{I_{m+n+1},0}=0,\ \ \
b_{I_{m+n+1},0}=c_{G_{m+\frac{1}{2}},0}+c_{G_{n+\frac{1}{2}},0},
\end{eqnarray*}
which imply
\begin{eqnarray}\label{(0.5,0,0)-GG-results}
&&a_{I_{k},0}=0,\ \ \
c_{G_{\frac{1}{2}},0}=\frac{b_{I_{0},0}+a_{L_{1},0}}{2},\ \ \forall\,\,k\in\Z.
\end{eqnarray}

Combining the identities presented in \eqref{(0.5,0,0)-LL-results}--\eqref{(0.5,0,0)-GG-results}, we can deduce
\begin{eqnarray}\label{(0.5,0,0)-results}
\begin{array}{lll}
&&d_{\0,0}(L_{k})=ka_{L_{1},0}L_{k}+kb_{L_{1},0}I_{k},\vs{6pt}\\
&&d_{\0,0}(G_{k+\frac{1}{2}})
=(ka_{L_{1},0}+\frac{b_{I_{0},0}+a_{L_{1},0}}{2})G_{k}
+(2kb_{L_{1},0}+d_{G_{\frac{1}{2}},0})H_{k+\frac{1}{2}},\vs{6pt}\\
&&d_{\0,0}(I_{k})=(ka_{L_{1},0}+b_{I_{0},0})I_{k},\ \ \ \
d_{\0,0}(H_{k+\frac{1}{2}})=(ka_{L_{1},0}+\frac{3b_{I_{0},0}+a_{L_{1},0}}{2})H_{k}.
\end{array}
\end{eqnarray}
It is easy to see that $d_{\0,0}=\ad(-a_{L_{1},0}L_{0}-b_{L_{1},0}I_{0})+ d_{G_{\frac{1}{2}},0}\partial_{4}+\frac{1}{2}\delta_{\lambda,0}b_{I_{0},0}
\partial_{3}$, which implies
\[\Der_{\0}({\mathfrak{L}^{\frac{1}{2}}_\lambda})
\cap\Der_{0}({\mathfrak{L}^{\frac{1}{2}}_\lambda})=
\Inn_{\0}({\mathfrak{L}^{\frac{1}{2}}_\lambda})
\cap\Inn_{0}({\mathfrak{L}^{\frac{1}{2}}_\lambda})
\oplus\C\delta_{\lambda,0}\partial_{3}\oplus\C\partial_{4}.\]
\end{proof}

\begin{lemm}\label{lemma(0.5,1,k)}
We have $\Der_{\1}({\mathfrak{L}^{\frac{1}{2}}_\lambda})
\cap\Der_{k+\frac{1}{2}}({\mathfrak{L}^{\frac{1}{2}}_\lambda})
=\Inn_{\1}({\mathfrak{L}^{\frac{1}{2}}_\lambda})
\cap\Inn_{k+\frac{1}{2}}({\mathfrak{L}^{\frac{1}{2}}_\lambda})$.
\end{lemm}
\begin{proof}
For any $d_{\1,k}\in\Der_{\1}({\mathfrak{L}^{\frac{1}{2}}_\lambda})
\cap\Der_{k+\frac{1}{2}}({\mathfrak{L}^{\frac{1}{2}}_\lambda})$, we can suppose
\begin{eqnarray}\label{eq(0.5,1,k)}
\begin{array}{ccccc}
d_{\1,k+\frac{1}{2}}(L_{n})\!\!\!&=&\!\!\!c_{L_{n},k+\frac{1}{2}}G_{n+k+\frac{1}{2}}
\!\!\!&+&\!\!\!d_{L_{n},k+\frac{1}{2}}H_{n+k+\frac{1}{2}},\vs{6pt}\\
d_{\1,k+\frac{1}{2}}(I_{n})\!\!\!&=&\!\!\!c_{I_{n},k+\frac{1}{2}}G_{n+k+\frac{1}{2}}
\!\!\!&+&\!\!\!d_{I_{n},k+\frac{1}{2}}H_{n+k+\frac{1}{2}},\vs{6pt}\\
d_{\1,k+\frac{1}{2}}(G_{n+\frac{1}{2}})
\!\!\!&=&\!\!\!a_{G_{n+\frac{1}{2}},k+\frac{1}{2}}L_{n+k+1}
\!\!\!&+&\!\!\!b_{G_{n+\frac{1}{2}},k+\frac{1}{2}}I_{n+k+1},\vs{6pt}\\
d_{\1,k+\frac{1}{2}}(H_{n+\frac{1}{2}})
\!\!\!&=&\!\!\!a_{H_{n+\frac{1}{2}},k+\frac{1}{2}}L_{n+k+1}
\!\!\!&+&\!\!\!b_{H_{n+\frac{1}{2}},k+\frac{1}{2}}I_{n+k+1}.
\end{array}
\end{eqnarray}
Using $[L_{n},L_{m}]=(n-m)L_{n+m}$ and \eqref{eq-derivation}, we have
\begin{eqnarray*}
\begin{array}{lll}
(n-m)c_{L_{m+n},k+\frac{1}{2}}
\!\!\!&=&\!\!\!(\frac{n-1}{2}-m-k)c_{L_{m},k+\frac{1}{2}}
-(\frac{m-1}{2}-n-k)c_{L_{n},k+\frac{1}{2}},\vs{6pt}\\
(n-m)d_{L_{m+n},k+\frac{1}{2}}
\!\!\!&=&\!\!\!\lambda(n+1)c_{L_{m},k+\frac{1}{2}}
+(\frac{n-1}{2}-m-k)d_{L_{m},k+\frac{1}{2}}\vs{6pt}\\
&&-\lambda(m+1)c_{L_{n},k+\frac{1}{2}}
-(\frac{m-1}{2}-n-k)d_{L_{n},k+\frac{1}{2}},
\end{array}
\end{eqnarray*}
which imply
\begin{eqnarray}\label{(0.5,1,k)-LL-results}
\begin{array}{lll}
&&c_{L_{m},k+\frac{1}{2}}=(1-\frac{m}{2k+1})c_{L_{0},k+\frac{1}{2}},\vs{6pt}\\
&&d_{L_{m},k+\frac{1}{2}}=-\lambda m(\frac{1}{k+\frac{1}{2}}+\frac{1}{2{(k+\frac{1}{2})}^2})c_{L_{0},k+\frac{1}{2}}
+(1-\frac{m}{2k+1})d_{L_{0},k+\frac{1}{2}}.
\end{array}
\end{eqnarray}
Acting $d_{\1,k+\frac{1}{2}}$ on both sides of $[L_{n},I_{m}]=(n-m)I_{n+m}$, we obtain
\begin{eqnarray*}
\begin{array}{lll}
(n-m)c_{I_{m+n},k+\frac{1}{2}}
\!\!\!&=&\!\!\!(\frac{n-1}{2}-m-k)c_{I_{m},k+\frac{1}{2}},\vs{6pt}\\
(n-m)d_{I_{m+n},k+\frac{1}{2}}\!\!\!&=&\!\!\!\lambda(n+1)c_{I_{m},k+\frac{1}{2}}
+(\frac{n-1}{2}-m-k)d_{I_{m},k+\frac{1}{2}}\vs{6pt}\\
&&-(m-2n-2k-1)c_{L_{n},k+\frac{1}{2}},
\end{array}
\end{eqnarray*}
from which we can deduce
\begin{eqnarray}\label{(0.5,1,k)-LI-results}
&&c_{I_{m},k+\frac{1}{2}}=0,\ \ \ \
d_{I_{m},k+\frac{1}{2}}=(2-\frac{2m}{2k+1})c_{L_{0},k+\frac{1}{2}}.
\end{eqnarray}
Using $[L_{n},H_{m+\frac{1}{2}}]=(\frac{n}{2}-m)H_{n+m+\frac{1}{2}}$ and \eqref{eq-derivation}, we obtain
\begin{eqnarray*}
\begin{array}{ccc}
(\frac{n-1}{2}-m)a_{H_{m+n+\frac{1}{2}},k+\frac{1}{2}}
\!\!\!&=&\!\!\!(n-m-k-1)a_{H_{m+\frac{1}{2}},k+\frac{1}{2}},\vs{6pt}\\
(\frac{n-1}{2}-m)b_{H_{m+n+\frac{1}{2}},k+\frac{1}{2}}
\!\!\!&=&\!\!\!(n-m-k-1)b_{H_{m+\frac{1}{2}},k+\frac{1}{2}},
\end{array}
\end{eqnarray*}
which imply
\begin{eqnarray}\label{(0.5,1,k)-LH-results}
&&a_{H_{m+\frac{1}{2}},k+\frac{1}{2}}=b_{H_{m+\frac{1}{2}},k+\frac{1}{2}}=0.
\end{eqnarray}
Using $[L_{n},G_{m+\frac{1}{2}}]=(\frac{n}{2}-m)G_{n+m+\frac{1}{2}}
+\lambda(n+1)H_{m+n+\frac{1}{2}}$ and \eqref{eq-derivation}, we obtain
\begin{eqnarray*}
&&\!\!\!\!\!\!\!\!\!
(\frac{n-1}{2}-m)a_{G_{m+n+\frac{1}{2}},k+\frac{1}{2}}
\!+\!\lambda(n+1)a_{H_{m+n+\frac{1}{2}},k+\frac{1}{2}}
\!=\!(n-m-k-1)a_{G_{m+\frac{1}{2}},k+\frac{1}{2}},\vs{6pt}\\
&&\!\!\!\!\!\!\!\!\!
(\frac{n-1}{2}-m)b_{G_{m+n+\frac{1}{2}},k+\frac{1}{2}}
\!+\!\lambda(n+1)b_{H_{m+n+\frac{1}{2}},k+\frac{1}{2}}
\!=\!(n-m-k-1)b_{G_{m+\frac{1}{2}},k+\frac{1}{2}}
\!+\!c_{L_{n},k+\frac{1}{2}},
\end{eqnarray*}
from which we can deduce
\begin{eqnarray}\label{(0.5,1,k)-LG-results}
&&a_{G_{m+\frac{1}{2}},k+\frac{1}{2}}=0,\ \ \ \ \
b_{G_{m+\frac{1}{2}},k+\frac{1}{2}}=\frac{2}{2k+1}c_{L_{0},k+\frac{1}{2}}.
\end{eqnarray}
By \eqref{(0.5,1,k)-LL-results}--\eqref{(0.5,1,k)-LG-results}, one has
$d_{\1,k+\frac{1}{2}}=\ad(\frac{1}{k+\frac{1}{2}}c_{L_{0},k
+\frac{1}{2}}G_{k+\frac{1}{2}}+(\frac{\lambda}
{{(k+\frac{1}{2})}^{2}}c_{L_{0},k+\frac{1}{2}}+
\frac{d_{L_{0},k+\frac{1}{2}}}{k+\frac{1}{2}})H_{k+\frac{1}{2}})$,
which implies $\Der_{\1}({\mathfrak{L}^{\frac{1}{2}}_\lambda})
\cap\Der_{k+\frac{1}{2}}({\mathfrak{L}^{\frac{1}{2}}_\lambda})
=\Inn_{\1}({\mathfrak{L}^{\frac{1}{2}}_\lambda})
\cap\Inn_{k+\frac{1}{2}}({\mathfrak{L}^{\frac{1}{2}}_\lambda})$.
\end{proof}

For any $d_{\1,0}\in\Der_{\1}({\mathfrak{L}^{\frac{1}{2}}_\lambda})
\cap\Der_{0}({\mathfrak{L}^{\frac{1}{2}}_\lambda})$, we have $d_{\1,0}=0$, which implies the following lemma.

\begin{lemm}\label{lemma(0.5,1,0)}
We have $\Der_{\1}({\mathfrak{L}^{\frac{1}{2}}_\lambda})
\cap\Der_{0}({\mathfrak{L}^{\frac{1}{2}}_\lambda})
=\Inn_{\1}({\mathfrak{L}^{\frac{1}{2}}_\lambda})
\cap\Inn_{0}({\mathfrak{L}^{\frac{1}{2}}_\lambda})=\{0\}$.
\end{lemm}

\section{Automorphism groups}

\indent\ \ \ \ \ \
Denote by $\Aut({\L})$ the automorphism group of ${\L}$ and $\mathfrak{I}$ the subgroup generated by $\{\exp(\alpha\ad I_{k})| \alpha\in\C, k\in\Z\}$. Denote by $\sigma$ the linear map from ${\L}$ to ${\L}$ satisfying
\begin{eqnarray}\label{aut-sigma}
\begin{array}{llllll}
&&\sigma(L_{k})=\epsilon{\alpha}^{k}L_{\epsilon k}+k{\alpha}^{k}\beta_{\epsilon}I_{\epsilon k},\vs{6pt}\\
&&\sigma(I_{k})={\alpha}^{k}\mu I_{\epsilon k},\ \ \ \
\sigma(H_{k+s})={\alpha}^{k}xH_{\epsilon(k+s)},\vs{6pt}\\
&&\sigma(G_{k+s})
=\frac{x}{\mu\epsilon}{\alpha}^{k}G_{\epsilon(k+s)}
+({\alpha}^{k}\gamma+\frac{2k(k-2s){\beta}_{\epsilon}{\alpha}^{k}x}{\mu})
H_{\epsilon(k+s)},
\end{array}
\end{eqnarray}
where $\epsilon=1$ or $-1$, $k\in\Z$, $\mu, \alpha, {\beta}_{\epsilon}, \gamma\in\C$, $\alpha\mu\neq0$ and $x^2={\alpha}^{2s}{\mu}^{3}$. Denote by $\mathfrak{G}$ the set generated (via composition of linear maps) by all such $\sigma$.

\begin{theo}\label{theorem-automorphism}

(1) $\mathfrak{I}$ is an abelian normal subgroup of $\Aut({\L})$.

(2) If $\lambda=0$, then $\mathfrak{G}$ is a subgroup of $\Aut({\L})$, and $\Aut({\L})=\mathfrak{I}\rtimes\mathfrak{G}$.

(3) If $\lambda\neq0$, $\mu=\epsilon=1$, then $\mathfrak{G}$ is a subgroup of $\Aut({\L})$, and $\Aut({\L})=\mathfrak{I}\rtimes\mathfrak{G}$.

(4) If $\lambda\neq0$, $\epsilon=-1$, then $\Aut({\L})=\mathfrak{I}$.

\end{theo}

Before proving Theorem \ref{theorem-automorphism}, we first give several lemmas.

\begin{lemm}\label{aut-lemma-001}
Denote by $\textbf{I}$ the vector space spanned by $\{I_{k}|k\in\Z\}$ over $\C$, we have $\sigma(I_{k})\in\textbf{I}$.
\end{lemm}
\begin{proof}
It is trivial according to $[G_{p},G_{q}]=I_{p+q}$ and \eqref{eq-automorphism}.
\end{proof}

\begin{lemm}\label{aut-lemma-002}
Replacing $\sigma$ with $\sigma\tau$ for some suitable $\tau\in\mathfrak{I}$ if necessary, and still denote as $\sigma$, one can assume
\begin{eqnarray}
&&\sigma(L_{m})=\epsilon {\alpha}^{m}L_{\epsilon m}+m{\alpha}^{m}\beta_{\epsilon}I_{\epsilon m},\ \ \ \
\sigma(I_{m})={\alpha}^{m}\mu I_{\epsilon m},
\end{eqnarray}
where $\alpha, \mu\in\C$, $\alpha\neq0$, $\mu\neq0$, $\epsilon=1$ or $-1$, $\beta_{\epsilon}\in\C$.
\end{lemm}
\begin{proof}
The restriction of $\sigma$ on ${(\L)}_{\0}$ is an element in $\Aut(W(a,b))$ for $a=0$, $b=-1$ in \cite{GJP}. This lemma follows immediately from Theorem $5.2$ of \cite{GJP}.
\end{proof}

\begin{lemm}\label{aut-lemma-003}
Denote by $\textbf{H}$ the vector space spanned by $\{H_{p}|p\in{\Z_s}\}$. Then $\sigma(H_{p})\in\textbf{H}$.
\end{lemm}
\begin{proof}
Considering the actions of $\sigma$ both sides of $[I_{n},G_{p}]=(n-2p)H_{n+p}$ and using Lemma \ref{aut-lemma-001}, we obtain this lemma.
\end{proof}

Now we can assume $\sigma(H_{p})=\sum_{q\in{\Z_s}}d_{H_{p},H_{q}}H_{q}$ and $\sigma(G_{p})=\sum_{q\in{\Z_s}}(c_{G_{p},G_{q}}G_{q}+d_{G_{p},H_{q}}H_{q})$. We give the definition of {\it level} of elements. Take $x=\sum_{p=b}^{t}a_{p}H_{p}\in\textbf{H}, (a_{t}\,a_{b}\neq0)$ for example: we say $t$ is the {\it top level} of $x$ and $b$ is the {\it bottom level} of $x$. By considering the top level and bottom level of $[\sigma(L_{m}),\sigma(H_{p})]=(\frac{m}{2}-p)\sigma(H_{m+p})$ and $[\sigma(L_n),\sigma(G_p)]=(\frac{n}{2}-p)\sigma(G_{p+n})+\lambda(n+1)\sigma(H_{n+p})$, together with Lemmas \ref{aut-lemma-002} and \ref{aut-lemma-003}, we immediately have the following two lemmas.
\begin{lemm}\label{aut-lemma-004}
$\sigma(H_{p})=d_{H_{p},H_{\epsilon p}}H_{\epsilon p}$.
\end{lemm}
\begin{lemm}\label{aut-lemma-005}
$\sigma(G_{p})=c_{G_{p},G_{\epsilon p}}G_{\epsilon p}+\sum_{q}d_{G_{p},H_{q}}H_{q}$.
\end{lemm}

\begin{lemm}\label{aut-lemma-006}
One has the following identities:
\begin{eqnarray}\label{aut-restriction of coefficients}
\begin{array}{llllll}
&&d_{H_{k+s},H_{\epsilon (k+s)}}={\alpha}^{k}d_{H_{s},H_{\epsilon s}},\vs{6pt}\\
&&c_{G_{p},G_{\epsilon p}}=\frac{1}{\mu\epsilon}d_{H_{p},H_{\epsilon p}},\ \ \
{d^{2}_{H_{s},H_{\epsilon s}}}={\alpha}^{2s}{\mu}^{3},\vs{6pt}\\
&&\lambda(\frac{1}{\mu}-1)=0,\ \ \,\lambda=0\ \ \ or\ \ \ \lambda\neq0,\ \,\mu=\epsilon=1,\vs{6pt}\\
&&d_{G_{k+s},H_{\epsilon(k+s)}}={\alpha}^{k}d_{G_{s},H_{\epsilon s}}+2k\epsilon {\alpha}^{k}{\beta}_{\epsilon}(k-2s)c_{G_{s},G_{\epsilon s}}.
\end{array}
\end{eqnarray}
\end{lemm}
\begin{proof}
Using $[L_{m},H_{p}]=(\frac{m}{2}-p)H_{m+p}$, we obtain
\begin{equation*}
{\alpha}^{m}d_{H_{p},H_{\epsilon p}}=d_{H_{m+p},H_{\epsilon(m+p)}},
\end{equation*}
which implies
\begin{eqnarray}\label{aut-lemma-005-eq-001}
d_{H_{k+s},H_{\epsilon (k+s)}}={\alpha}^{k}d_{H_{s},H_{\epsilon s}},\ \ \forall\,\,k\in\Z.
\end{eqnarray}
Using $[I_{m},G_{p}]=(m-2p)H_{m+p}$, one has
\begin{equation*}
{\alpha}^{m}\mu\epsilon c_{G_{p},G_{\epsilon p}}=d_{H_{m+p},H_{\epsilon(m+p)}},
\end{equation*}
from which we can deduce
\begin{eqnarray}\label{aut-lemma-005-eq-002}
c_{G_{p},G_{\epsilon p}}=\frac{1}{\mu\epsilon}d_{H_{p},H_{\epsilon p}}.
\end{eqnarray}
According to $[G_{p},G_{q}]=I_{p+q}$, one has
\begin{equation*}
c_{G_{p},G_{\epsilon p}}c_{G_{q},G_{\epsilon q}}={\alpha}^{p+q}\mu,
\end{equation*}
which together with \eqref{aut-lemma-005-eq-002} gives
\begin{eqnarray}\label{aut-lemma-005-eq-003}
{d_{H_{s},H_{\epsilon s}}}^{2}={\alpha}^{2s}{\mu}^{3}.
\end{eqnarray}

Acting $\sigma$ on both sides of $[L_{0},G_{p}]=-pG_{p}+\lambda H_{p}$, we obtain
\begin{equation*}
\sum_{r=b}^{t}d_{G_{p},H_{r}}(\epsilon r-p)H_{r}=\lambda(\frac{1}{\mu}-1)d_{H_{p},H_{\epsilon p}}H_{\epsilon p}.
\end{equation*}
Compare the top and bottom level of both sides, one can deduce
\begin{eqnarray}\label{aut-lemma-005-eq-004}
&&\lambda(\frac{1}{\mu}-1)=0,\ \ \
\sigma(G_{p})=c_{G_{p},G_{\epsilon p}}G_{\epsilon p}+d_{G_{p},H_{\epsilon p}}H_{\epsilon p}.
\end{eqnarray}
Using $[L_{m},G_{p}]=(\frac{m}{2}-p)G_{m+p}+\lambda(m+1)H_{m+p}$, one has
\begin{eqnarray}\label{aut-lemma-005-eq-005}
\begin{array}{lll}
&&(\frac{m}{2}-p)d_{G_{m+p},H_{\epsilon(m+p)}}
+\lambda(m+1)d_{H_{m+p},H_{\epsilon(m+p)}}\vs{6pt}\\
&&={\alpha}^{m}\lambda(m+\epsilon)c_{G_{p},G_{\epsilon p}}+{\alpha}^{m}(\frac{m}{2}-p)d_{G_{p},H_{\epsilon p}}
+\epsilon {\alpha}^{m}m{\beta}_{\epsilon}(m-2p)c_{G_{p},G_{\epsilon p}}.
\end{array}
\end{eqnarray}
Taking $m=2p$ in \eqref{aut-lemma-005-eq-005}, one has
\begin{equation*}
\lambda(2(1-\mu\epsilon)p+\epsilon-\mu\epsilon)=0,\ \ \forall\,\,p\in{\Z_s},
\end{equation*}
which implies
\begin{eqnarray}\label{aut-lemma-005-eq-006}
&&\lambda=0\ \ \
or\ \ \ \lambda\neq0,\ \ \mu=\epsilon=1.
\end{eqnarray}
According to \eqref{aut-lemma-005-eq-006}, \eqref{aut-lemma-005-eq-005} becomes
\begin{equation*}
d_{G_{m+p},H_{\epsilon(m+p)}}={\alpha}^{m}d_{G_{p},H_{\epsilon p}}+2m\epsilon {\alpha}^{m}{\beta}_{\epsilon}(m-2p)c_{G_{p},G_{\epsilon p}},
\end{equation*}
from which we can deduce $d_{G_{k+s},H_{\epsilon(k+s)}}={\alpha}^{k}d_{G_{s},H_{\epsilon s}}+2k\epsilon {\alpha}^{k}{\beta}_{\epsilon}(k-2s)c_{G_{s},G_{\epsilon s}}$.
\end{proof}

\noindent{\it Proof of Theorem \ref{theorem-automorphism}} (1) For any $\sigma\in\Aut({\L})$, we know $\sigma(\ad I_{k}){\sigma}^{-1}=\ad\sigma(I_{k})$. According to Lemma \ref{aut-lemma-001}, we know that $\mathfrak{I}$ is an abelian normal subgroup of $\Aut({\L})$. And Theorem \ref{theorem-automorphism} (2)--(4) follows from Lemmas \ref{aut-lemma-002}--\ref{aut-lemma-006} immediately.\hfill$\Box$

\end{document}